\documentclass[11pt,reqno]{amsart}
\usepackage{fullpage,enumerate}
\usepackage{amsfonts}
\usepackage[all]{xy}
\usepackage{amssymb}
\usepackage{comment}
\usepackage{times}
\usepackage{graphicx}
\usepackage{amsmath,bm}
\usepackage[hidelinks]{hyperref}
\usepackage{tikz-cd}
\usepackage{tikz}
\usepackage{appendix}
\usepackage{color}\usepackage{float}
\usepackage[margin=1.3in,marginparwidth=1.1in,marginparsep=0.1in]{geometry}
\usepackage{marginnote,color}
\usepackage{import}
\usepackage[normalem]{ulem}
\usepackage{graphicx}

\usepackage{mathrsfs}
\usepackage{geometry}
\usepackage{epsfig,amscd,amsthm,mathtools,fourier}
\numberwithin{equation}{section}

\newtheorem{thm}{Theorem}[section]
\newtheorem{cor}[thm]{Corollary}

\newtheorem*{Thm}{Main Theorem}
\newtheorem*{Cor}{Corollary}

\theoremstyle{definition}

\newtheorem{defn}[thm]{Definition}

\newcommand{\trop}{\mathrm{trop}}

\newcommand{\rank}{{\rm rank}}

\newcommand{\val}{{\rm val}}

\newcommand{\Lin}{{\rm Lin}}
\newcommand{\Spec}{{\rm Spec}}

\newcommand{\RR}{{\mathbb R}}
\newcommand{\ZZ}{{\mathbb Z}}
\newcommand{\PP}{{\mathbb P}}

\newcommand{\GG}{{\mathbb G}}

\newcommand{\CM}{{\mathcal M}}

\newcommand{\CF}{{\mathcal F}}

\newcommand{\Star}{{\rm Star}}

\newcommand{\ord}{{\rm ord}}

\def\:{\colon}
\def\val{\nu}

\newcommand{\cC}{{\mathcal{C}}}

\newcommand{\red}{\mathrm{red}}

\theoremstyle{remark}
\newtheorem{rem}[thm]{Remark}

\makeatletter
\DeclareRobustCommand{\cev}[1]{%
  {\mathpalette\do@cev{#1}}%
}
\newcommand{\do@cev}[2]{%
  \vbox{\offinterlineskip
    \sbox\z@{$\m@th#1 x$}%
    \ialign{##\cr
      \hidewidth\reflectbox{$\m@th#1\vec{}\mkern4mu$}\hidewidth\cr
      \noalign{\kern-\ht\z@}
      $\m@th#1#2$\cr
    }%
  }%
}
\makeatother

\title{Balancing properties of tropical moduli maps}
\author{Karl Christ, Xiang He, and Ilya Tyomkin}
\thanks{
KC was partially supported by NSF FRG grant DMS–2053261, a Minerva Short-Term Research Grant, and the Center for Advanced Studies at BGU. XH is supported by the National Key R$\&$D Program of China 2022YFA1007100, the NSFC grant 12301057, and the ERC Consolidator Grant 770922 - BirNonArchGeom.\\
2020 Mathematics Subject Classification: Primary – 14T20; Secondary – 14H10
}

\address[Christ]{Dipartimento di Matematica\\
	Università di Torino\\Via Carlo Alberto 10 \\10123 Turin\\  Italy }\email{karl.christ@unito.it}

\address[He]{Yau Mathematical Sciences Center\\ Shuangqing Complex Building, Tsinghua University\\ Haidian District, Beijing\\ 100084\\ China}\email{xianghe@mail.tsinghua.edu.cn}

\address[Tyomkin]{Department of Mathematics\\
	Ben-Gurion University of the Negev\\P.O.Box 653 \\Be'er Sheva\\ 84105\\  Israel}\email{tyomkin@math.bgu.ac.il}

\begin{document}

\begin{abstract}
Given a family of parameterized algebraic curves over a strictly semistable pair, we show that the simultaneous tropicalization of the curves in the family forms a family of parameterized tropical curves over the skeleton of the strictly semistable pair. We show that the induced tropical moduli map satisfies a certain balancing condition, which allows us to describe properties of its image and deduce a new liftability criterion.
\end{abstract}
	
\maketitle
	
\setcounter{tocdepth}{1}

\section{Introduction}

Starting with the seminal work of Mikhalkin, \cite{Mik05}, tropical geometry emerged as a very useful tool in many aspects of algebraic geometry, including enumerative problems, questions about the geometry of various moduli spaces, etc. Tropical varieties are polyhedral objects controlling degenerations of algebraic varieties, and tropical methods in algebraic geometry can be considered as a refined version of degeneration and deformation-theoretic techniques. 

Most applications of tropical geometry build upon so-called tropicalization and lifting results. Roughly speaking, tropicalization is a kind of linearization procedure associating a tropical object to an algebraic one, and lifting results aim at realizing a given tropical object as the tropicalization of an algebro-geometric one. There exists no functorial way to tropicalize algebraic geometry. Yet, various useful tropicalization constructions have been developed, see, e.g., \cite{Mik05, Tyo12, ACP15, ABBR15, Tyo17, ACGS, CCUW, Ran22, MUW21, CGM, CHT23}. In particular, in \cite{CHT23}, we introduced tropicalizations of one-parameter families of curves in toric varieties and studied their properties. As an application, in {\em loc. cit.} we resolved the Severi problem in arbitrary characteristic.

In general, lifting results are more challenging to achieve as they are based on deformation theory, which often is obstructed and difficult to control. Surprisingly, the basic balancing properties of tropicalizations of families established in \cite{CHT23} allowed us to prove non-trivial lifting results without any deformation theory. The goal of the current paper is to extend the theory initiated in {\em loc. cit.} to families of curves over higher-dimensional bases. In \cite{CHT24b, CHT25}, we use the results of the current paper to prove the irreducibility of Hurwitz spaces in arbitrary characteristic, and to count moduli of curves of given genus that can be realized as a curve in a given linear system on a toric surface. We believe that the results of the current paper will find further applications in the future.

Let us describe the main results of the paper in more detail. Suppose we are given a family of smooth curves $\cC \to U$ with marked points together with a map $f\:\cC \to X$ to a toric variety $X$. After replacing $U$ with an appropriate alteration, we construct a family of stable curves whose base $B$ admits a natural tropicalization $\Lambda$ in the sense of Gubler, Rabinoff, and Werner \cite{GRW16}. We then construct a family of parameterized tropical curves $\left(\Gamma_\Lambda \to \Lambda, h\:\Gamma_\Lambda \to \RR^n\right)$ such that for any $b\in B$ the tropicalization of $f_b\:\cC_b\to X$ is the fiber $h_q\:\Gamma_q\to\RR^n$, where $q=\trop(b)\in\Lambda$. 
The family of parameterized tropical curves $\left(\Gamma_\Lambda \to \Lambda, h\:\Gamma_\Lambda \to \RR^n\right)$ induces a natural map to the moduli space of stable parameterized tropical curves of appropriate degree $\nabla$, genus $g$, and number $n$ of contracted legs \[\alpha \colon \Lambda \to M_{g,n,\nabla}^\trop,\] 
given by $q\mapsto (\Gamma_q,h_q)$; see \S~\ref{subsec:modulioftrcr} for the description  of the moduli space $M_{g,n,\nabla}^\trop$ and \S~\ref{subsubsec:family of parameterized tropical curves} for the definition of families of parameterized tropical curves and the construction of the induced map $\alpha$. 

It is natural to expect this map to be itself a tropical map, i.e., a piecewise integral affine map satisfying certain balancing properties. Our main result provides first evidence for this expectation. Namely, we show that the map $\alpha$ is a piecewise integral affine map, and that it satisfies natural balancing properties along two types of strata of small codimension: the maximal strata classifying weightless and $3$-valent parameterized tropical curves, and their codimension-one boundary strata classifying weightless parameterized tropical curves which are {\em almost $3$-valent}, i.e., $3$-valent except for a unique $4$-valent vertex. Here is our main result (see Theorem~\ref{thm:main thm} for the precise formulation).

\begin{Thm}
    Let $(B^0,H^0)$ be a strictly semistable pair over the ring of integers of a non-Archi\-medean field $K$, and $(\cC^0\rightarrow B^0,\sigma^0_\bullet)$ a split family of stable marked curves over $B^0$. Let $B$, $H$ and $\cC$ be the generic fibers of $B^0$, $H^0$, and $\cC^0$, respectively. Suppose $\cC$ is smooth over $B':=B\backslash H$. Let $X$ be a toric variety and $f \colon \cC|_{B'} \rightarrow X$ a morphism, such that for any $b\in B'$, the restriction $f|_b\:\cC_b\to X$ is torically transverse, i.e.,  $f|_b^{-1}$ preserves the codimension of the torus orbits. Then there exists a family of parameterized tropical curves $h\: \Gamma_\Lambda\to N_\RR$ over the Gubler-Rabinoff-Werner tropicalization $\Lambda$ of the pair $(B^0,H^0)$ such that 
    \begin{enumerate}
        \item For any $\eta \in B'(K)$, the fiber of $(\Gamma_\Lambda, h)$ over $\trop(\eta)$ is the tropicalization of $f\:\cC_\eta\to X$;
        \item For any face $W$ of $\Lambda$ of codimension one (resp. of codimension at least two) such that the tropical curves over the interior of $W$ are weightless and $3$-valent except for at most one $4$-valent vertex, the induced map $\alpha\:\Lambda\to M^{\trop}_{g,n,\nabla}$ is either harmonic (resp. quasi-harmonic) or locally combinatorially surjective along $W$.
    \end{enumerate}
\end{Thm}

Here, {\em quasi-harmonic} means that a positive linear combination of the images of the primitive normal vectors to a stratum is contained in the image of the stratum, and {\em harmonic} requires that the coefficients in the linear combination can be chosen to be $1$; see Definition~\ref{defn:balancing}. {\em Locally combinatorially surjective}, on the other hand, means that the image of $\alpha$ contains points in each of the weightless and $3$-valent strata adjacent to the weightless and almost $3$-valent stratum; see Definition~\ref{def:harloccombsurj}. 

As a consequence of the Main Theorem we establish the following corollary (see Corollary~\ref{cor:main}) that provides a new relative realizability criterion for (parameterized) tropical curves.

\begin{Cor}
    Let $X$ be a toric variety and $f\colon \cC \rightarrow X$ a family of parameterized curves defined over the field $K$. Assume that the base $B$ of the family is quasi-projective, and consider the set of tropicalizations
    $$\Sigma:=\left\{\left[ \trop(f_\eta, C_\eta)\right]\; :\; \eta\in B(K) \right\}\subseteq M^{\trop}_{g,n,\nabla}.$$
    Let $M_{[\Theta]}\subset M^{\trop}_{g,n,\nabla}$ be a stratum. Then,
    \begin{enumerate}
        \item The closure $\overline\Sigma\subseteq M^{\trop}_{g,n,\nabla}$ is the image of a rational polyhedral complex $\Lambda$ under a piecewise integral affine map;
        \item If $M_{[\Theta]}$ is weightless and $3$-valent and $\dim(\overline{\Sigma}\cap M_{[\Theta]})=\dim M_{[\Theta]}$, then $\overline M_{[\Theta]}\subseteq \overline{\Sigma}$;
        \item If $M_{[\Theta]}$ is weightless and almost $3$-valent and $\overline{\Sigma}$ contains one of its adjoint weightless and $3$-valent strata, then $\overline{\Sigma}$ contains all other such strata as well. 
    \end{enumerate}
\end{Cor}

In \cite{CHT24b} we apply this corollary to prove the irreducibility of Severi varieties parametrizing integral curves of a given geometric genus on certain polarized toric surfaces. We then deduce the irreducibility of the classical Hurwitz spaces in any characteristic. More precisely, in \emph{loc. cit.}, we apply the corollary to the universal family over an irreducible component $B$ of the Severi variety. By imposing $\dim(B)$ point constraints, whose tropicalizations are in tropically general position, one can easily show that there exists a weightless and $3$-valent combinatorial type $\Theta$ such that $\dim(\overline{\Sigma}\cap M_{[\Theta]})=\dim M_{[\Theta]}$. On the other hand, one can show that there exists a weightless and $3$-valent type $\Theta_0$ such that $\overline M_{[\Theta_0]}\subseteq \overline{\Sigma}_0$, where $\Sigma_0$ corresponds to a unique irreducible component $B_0$ of the Severi variety. Thus, Assertion (2) of the corollary, reduces the irreducibility problem to a purely combinatorial question reminiscent of establishing connectivity of the moduli space $M^{\trop}_{g,n,\nabla}$ in codimension one. Namely, is there a sequence of (almost) $3$-valent strata adjacent to each other starting at $M_{[\Theta]}$ and ending at $M_{[\Theta_0]}$?  

Another application of the results of the current paper will be a computation of the number of moduli of curves in (irreducible) Severi varieties. More precisely, normalizing the universal family over a Severi variety $V$ induces a map from $V$ to the moduli space $\CM_g$ of smooth genus $g$ curves. Similarly, on the tropical side, we obtain a natural map from the moduli space $M_{g,\nabla}^\trop$ to the moduli space $M_g^{\trop}$ of stable tropical curves of genus $g$. Our main result, together with the results of \cite{ACP15} for $M_g^{\trop}$, allows to relate the tropical and the algebro-geometric side of this picture. In a forthcoming paper \cite{CHT25} we use it to establish the dimension of the image of Severi varieties in the moduli space $\CM_g$.

\medskip

\textbf{Acknowledgements:} We thank an anonymous referee for many thoughtful comments.

\section{Notation and terminology}
Throughout this paper, we work over a valued field $(K, \nu)$, which is the algebraic closure of a complete discretely valued field $F$ with algebraically closed residue field. The ring of integers of $K$ is denoted by $K^0$, the maximal ideal by $K^{00}$, and the residue field by $\widetilde K$; similarly for $F$.

Let $N\simeq \ZZ^n$ be a lattice, and $N_\RR:=N\otimes_\ZZ\RR$ the corresponding vector space. By a {\em polyhedron} in $N_\RR$ we mean the intersection of finitely many affine half-spaces. A polyhedron is said to have {\em rational slopes} if the presentation above can be chosen such that all half-spaces are given by integral normal vectors. For a subset $P\subset N_\RR$, we denote by $\Lin(P)$ the linear subspace acting simply transitively on the affine hull of $P$. If $P$ is a polyhedron with rational slopes, then $\Lin(P)$ admits a natural integral structure induced from $N_\RR$, i.e., the lattice $N\cap \Lin(P)$ of integral vectors. By an {\em integral structure} on $P$, we mean the integral structure on $\Lin(P)$.

\subsection{Polyhedral complexes}

\begin{defn}\label{defn:abstract poly complex}
    An \textit{(abstract) polyhedral complex with integral structure} is a connected topological space $\Lambda$ together with a finite set $\mathcal F(\Lambda)$ of closed subsets, called {\em faces of $\Lambda$}, and a map $\mu_W\colon W\rightarrow N_W\otimes\RR$ for every $W\in\mathcal F(\Lambda)$ such that 
    \begin{enumerate}
        \item $N_W$ is a lattice of finite rank for any face $W\in\mathcal F(\Lambda)$;
        \item The map $\mu_W$ is a homeomorphism from $W$ to a polyhedron with rational slopes of dimension $\rank(N_W)$. By abuse of language, we will address $W$ itself as a polyhedron and speak about its faces, interior, dimension, integral structure, etc.;
        \item $\Lambda=\bigsqcup_{W\in\mathcal F(\Lambda)}W^\circ$, where $W^\circ$ denotes the interior of $W$;
        \item For any $W,W'\in\mathcal F(\Lambda)$, the intersection $W\cap W'$ is a union of faces of $W$ and of $W'$. And each face in the intersection belongs to $\mathcal F(\Lambda)$;        
        \item If $W\in\mathcal F(\Lambda)$ and $W'\in\mathcal F(\Lambda)$ is a face of $W$, then the integral structure on $W'$ is the restriction of the integral structure on $W$, i.e., the linear part of $\mu_W\circ \mu_{W'}^{-1}$ takes $N_{W'}$ to a saturated sublattice in $N_W$.
    \end{enumerate}
    We say that $\Lambda$ is {\em rational} if each $\mu_W(W)$ has rational vertices in $N_W\otimes\RR$. All polyhedral complexes considered in this paper will be assumed to have integral structures, and we will call them polyhedral complexes for short.
\end{defn}
\begin{rem}
    Axiom (2) implies that each $W\in\mathcal F(\Lambda)$ has a structure of a polyhedron with rational slopes, axioms (3) and (4) mean that $\Lambda$ is ``glued'' from polyhedra along their faces, and axiom (5) ensures that the integral structures on the polyhedra are compatible with the gluing.
\end{rem}

\begin{defn}
    A map $\beta\colon \Lambda\rightarrow \Lambda'$ is called \textit{piecewise integral affine} if for each face $W\in\mathcal F(\Lambda)$ there is a face $W'\in\mathcal F(\Lambda')$ such that $\beta(W)\subseteq W'$ and $\beta|_W\colon W\to W'$ is integral affine.
\end{defn}

For a face $W\in\mathcal F(\Lambda)$, let $\{W_i\}_i$ be the set of faces in $\mathcal F(\Lambda)$ containing $W$ as a codimension-one face. Then, for each $i$, the lattice $N_W$ is naturally identified with a saturated sublattice of $N_{W_i}$ of corank one. Let $\vec e_i+N_W$ be the generator of $N_{W_i}/N_W$ for which $\mu_{W_i}(W_i)\subseteq \mu_{W_i}(W)+\mathbb R_{\geq 0}\vec e_i$, and denote by $\Star(W)$ the collection $\{\vec e_i+N_W\}$. We will often consider a small open neighborhood of $W^\circ$ in $W^\circ\cup\left(\bigcup_iW_i^\circ\right)$. To distinguish this set from $\Star(W)$, we denote it by $\Star(W^\circ)$.

\begin{defn}\label{defn:balancing}
    Let $\Lambda$ be a polyhedral complex and $W\in \mathcal F(\Lambda)$ a face. A piecewise integral affine map $\beta\colon\Lambda\rightarrow \RR^n$ is called \textit{quasi-harmonic} at $W$ if 
    $$ \sum_{\vec e+N_W\in\Star(W)}a_{\vec e}\cdot \frac{\partial \beta}{\partial \vec e}\in\Lin(\beta(W))$$
    for some positive integers $a_{\vec e}$. It is called \textit{harmonic} at $W$ if one can choose all $a_{\vec e}=1$. Finally, the map $\beta$ is called {\em (quasi-)harmonic} if it is (quasi-)harmonic at all codimension-one faces. 
\end{defn}

\begin{rem}
    Plainly, being (quasi-)harmonic at $W$ depends only on the restriction of the map $\beta$ to $\Star(W^\circ)$.
\end{rem}

\subsection{Tropical curves} We next recall some conventions concerning tropical curves and their moduli from \cite{CHT23,CHT22}, and extend the notion of one-parameter families of parameterized tropical curves therein to arbitrary dimension of the base. 
\subsubsection{Abstract and parameterized tropical curves}
A tropical curve $\Gamma$ consists of an underlying finite graph $\GG$, a {\em length function} $\ell\colon E(\mathbb G)\rightarrow \mathbb R_{>0}$ and a {\em weight (or genus) function} $g\colon V(\mathbb G)\rightarrow \mathbb Z_{\geq 0}$.
Here $V(\GG)$ denotes the set of vertices of $\GG$, and  $E(\GG)$ the set of edges.
The graph $\GG$ is furthermore equipped with an ordered set of half-edges, which we call {\em legs}. The set of legs is denoted by $L(\GG)$, and we define $\overline E(\mathbb G):=E(\GG)\cup L(\GG)$. We extend $\ell$ to the legs $l \in L(\GG)$ by setting $\ell(l):=\infty$.

We identify tropical curves with their geometric realization as polyhedral complexes, without necessarily distinguishing the two viewpoints. That is, we may divide the loops at the middle point, and then identify the edges of $\GG$ with bounded closed intervals of the corresponding lengths, and the legs of $\GG$ with semi-bounded closed intervals.

For $e\in\overline E(\GG)$, we write $\vec e$ to indicate a choice of orientation on $e$. If $e$ is a leg, then it will {\em always} be oriented away from the vertex, while bounded edges will be considered with both possible orientations. 
In particular, the elements in $\Star(v)$ are in one-to-one correspondence with the collection of oriented edges and legs having $v$ as their tail.
The number of elements in $\Star(v)$ is called the {\em valence} of $v$. The {\em genus} of $\Gamma$ is defined to be $g(\Gamma) = g(\GG) :=1-\chi(\mathbb G)+\sum_{v\in V(\mathbb G)}g(v)$, where $\chi(\GG):=b_0(\GG)-b_1(\GG)$ is the Euler characteristic of $\GG$. 

A tropical curve $\Gamma$ is called {\em weightless} if the weight function is identically zero and {\em stable} if $|\Star(v)|+ 2 g(v)\ge 3$ for every vertex $v\in V(\GG)$.

To define {\em parameterized tropical curves}, we fix a lattice $N$ and set $N_\RR:=N\otimes\RR$. 
Then a parameterized tropical curve is, by definition, a harmonic piecewise integral affine map $h\colon \Gamma\rightarrow N_\mathbb R$ from a tropical curve $\Gamma$ (considered as a polyhedral complex) to $N_\mathbb R$. 
The {\em combinatorial type} $\Theta$ of a parameterized tropical curve is defined to be the underlying weighted graph $\mathbb G$ together with the collection of slopes $\frac{\partial h}{\partial \vec e}$ for $e \in \overline E(\GG)$.
The {\em extended degree} $\overline\nabla$ of a parameterized tropical curve consists of the sequence of slopes $\left(\frac{\partial h}{\partial \vec l}\right)_{l\in L(\GG)}$, and the {\em degree} $\nabla$ of the subsequence $\left(\frac{\partial h}{\partial \vec l_i}\right)$ of non-zero slopes.

\subsubsection{Moduli of parameterized tropical curves}\label{subsec:modulioftrcr}

Fixing the genus $g$, degree $\nabla$, and number $n$ of contracted legs, we denote by $M_{g,n, \nabla}^{\trop}$ the moduli space of parameterized stable tropical curves with these invariants. We will always assume that the first $n$ legs $l_1,\dotsc, l_n$ are the contracted ones.
Then the space $M_{g,n, \nabla}^{\trop}$ is a generalized polyhedral cone complex with integral structure in the sense of \cite{ACP15}, whose strata $M_{[\Theta]}$ are indexed by combinatorial types $\Theta$ with the fixed invariants. If $\GG$ is the underlying graph of the combinatorial type $\Theta$, then $M_{[\Theta]} = M_{\Theta}/\mathrm{Aut}(\Theta)$, where $\mathrm{Aut}(\Theta)$ denotes the automorphisms of the combinatorial type. Here $M_{\Theta}$ is the interior of a polyhedron  $\overline M_\Theta$ in $\RR^{|E(\GG)|} \times N_\RR^{|V(\GG)|}$ and it parameterizes tropical curves $h\: \Gamma \to N_{\RR}$ of type $\Theta$: the $e$-coordinate for an edge $e \in E(\GG)$ encodes the length $\ell(e)$ of $e$; the $v$-coordinate for a vertex $v \in V(\GG)$ encodes the coordinate of $h(v) \in N_\RR$.

We call the stratum $M_\Theta$ {\em weightless and $3$-valent} if so is the combinatorial type $\Theta$, and we call it {\em weightless and almost $3$-valent} if $\Theta$ is weightless and $3$-valent except for a unique $4$-valent vertex. If $M_\Theta$ is weightless and almost $3$-valent it is contained in the closure of at most $3$ strata, all of which are weightless and $3$-valent. Furthermore, it has codimension one in each of them. 

\begin{rem}
    The nice strata (resp. simple walls) that are used in \cite{CHT23} are particular cases of weightless and $3$-valent (resp. almost $3$-valent) strata, where the strata are in addition required to be regular, i.e., of the expected dimension. 
\end{rem}

\subsubsection{Families of parameterized tropical curves over polyhedral complexes} \label{subsubsec:family of parameterized tropical curves}
In this subsection, we generalize the notion of one-parameter families of parameterized tropical curves in \cite[\S 3.1.3]{CHT23} to arbitrary dimension. 
Let $\Lambda$ be a polyhedral complex. Consider a datum (\dag) consisting of the following:

\begin{itemize}
\item an extended degree $\overline\nabla$;
\item a combinatorial type $\Theta_W=\left(\GG_W,\left(\frac{\partial h}{\partial\vec\gamma}\right)\right)$ of extended degree $\overline\nabla$ for each $W \in \mathcal F(\Lambda)$;
\item an integral affine function $\ell_W(\gamma,\cdot)\:W\to \RR_{\ge 0}$ for each $W \in \mathcal F(\Lambda)$ and $\gamma\in E(\GG_W)$;
\item an integral affine function $h_W (u,\cdot):W\to N_\RR$ for each $W \in \mathcal F(\Lambda)$ and $u\in V(\GG_W)$;
\item a weighted contraction $\phi_{W',W}\: \mathbb{G}_{W'} \to \mathbb{G}_W$ preserving the order of the legs for each pair of faces $W,W'$ such that $W$ is a face of $W'$.
\end{itemize}
For any $W \in \mathcal F(\Lambda)$ and any $q\in W^\circ$, set $\Gamma_q:=(\mathbb{G}_W, \ell_q)$, where $\ell_q\:E(\GG_W)\to \RR_{\ge 0}$ is the function defined by $\ell_q(\gamma):=\ell_W(\gamma,q)$. Denote by $h_q\colon \Gamma_q\rightarrow N_\mathbb R$ the unique piecewise affine map for which $\left(\frac{\partial h}{\partial \vec l_i}\right)=\overline\nabla$$h_q(u)=h_W(u,q)$ for all $u\in V(\GG_W)$ and $\left(\frac{\partial h}{\partial \vec l_i}\right)=\overline\nabla$.

\begin{defn}\label{defn:family of tropical curves}
    Let $\Lambda$ be a polyhedral complex. We say that a datum (\dag) is a {\em family of parameterized tropical curves over $\Lambda$} if the following compatibilities hold for all pairs $W,W' \in\Lambda$ such that $W$ is a face of $W'$, and all $q\in W^\circ$:
\begin{enumerate}
\item $(\Gamma_q,h_q)$ is a parameterized tropical curve of combinatorial type $\Theta_W$;
\item  $\ell_W(\phi_{W',W}(\gamma),q)=\ell_{W'}(\gamma,q)$ for all $\gamma\in E(\GG_{W'})$;
\item $h_W(\phi_{W',W}(u),q)=h_{W'}(u,q)$ for all $u\in V(\GG_{W'})$.
\end{enumerate}
A family of parameterized tropical curves over $\Lambda$ will be denoted by $h\:\Gamma_\Lambda\to N_\RR$ or $(\Gamma_\Lambda,h)$. The tropical curve $(\Gamma_q,h_q)$ will be referred to as the \textit{fiber} of $(\Gamma_\Lambda,h)$ over $q\in \Lambda$.
\end{defn}

\begin{rem}\label{rem:tropfamother}
    The version of tropical families and tropicalization we use in the current paper is natural when working with families of parametrized curves over non-Archimedean fields as the tropicalization of the base embeds naturally as a skeleton in the Berkovich analytification of the base, and similarly for the fibers. There exists another version of tropicalization arising naturally when working in the logarithmic category, see, e.g., \cite{ACGS, CCUW, R17, Ran22}. The logarithmic tropicalization is more general, but the base and the fibers belong to the category of (generalized) polyhedral \emph{cone} complexes. The relation between the two is given by considering the cones over the faces of $\Lambda$ and the cones over the tropical curves in the family. Notice also that the parametrization maps we consider in this paper are particular cases of (stable) maps in {\em loc. cit.}
\end{rem}

\begin{defn}
Let $\Lambda$ be a polyhedral complex, and $\alpha\:\Lambda\to M_{g, n, \nabla}^\trop$ a continuous map. We say that $\alpha$ is {\em piecewise integral affine} if for any $W\in\Lambda$ the restriction $\alpha|_W$ lifts to an integral affine map $W\to \overline M_\Theta$ for some combinatorial type $\Theta$.
\end{defn}
Let $\nabla$ be the degree associated to the extended degree $\overline\nabla$ by removing the zero slopes. Any family of parameterized tropical curves $h\:\Gamma_\Lambda \to N_\RR$ induces a piecewise integral affine map $\alpha\: \Lambda \to M_{g, n, \nabla}^{\trop}$ by sending $q \in \Lambda$ to the point parameterizing the isomorphism class of the stabilization of the fiber $h_q\: \Gamma_q \to N_\RR$. Furthermore, $\alpha$ lifts to an integral affine map from the interior of each face of $\Lambda$ to the corresponding stratum $M_\Theta$. The following definition is a generalization of \cite[Definition 3.4]{CHT23}.

\begin{defn}\label{def:harloccombsurj}
Let $\Lambda$ be a polyhedral complex, $\alpha\:\Lambda\to M_{g, n, \nabla}^\trop$ a piecewise integral affine map, $W$ a face of $\Lambda$, and $\Theta$ a combinatorial type such that $\alpha(W^\circ)\subset M_{[\Theta]}$.
\begin{enumerate}
 \item Suppose $\alpha(\Star(W^\circ))\subset M_{[\Theta]}$. We say that $\alpha$ is {\em (quasi-)harmonic} at $W$ if $\alpha_{|_{\Star(W^\circ)}}$ lifts to a (quasi-)harmonic map $\Star(W^\circ)\to M_\Theta$.
 \item Suppose $\alpha(\Star(W^\circ))\nsubseteq  M_{[\Theta]}$. We say that $\alpha$ is \textit{locally combinatorially surjective} at $W$ if for any $\overline M_{[\Theta]}\subseteq \overline M_{[\Theta']}$, we have $\alpha(\Star(W^\circ))\cap M_{[\Theta']}\neq \emptyset $.
\end{enumerate}
\end{defn}

\subsection{Families of curves}

\label{subsec: Families of curves}

A \emph{family of curves} is a flat, projective morphism of finite presentation and relative dimension one; cf. \cite[\S~2.1]{CHT22}. In our setting it often comes with \emph{marked points}, that is, an ordered collection of disjoint sections contained in the smooth locus of the family. A family of curves with marked points is \emph{prestable} if its geometric fibers have at-worst-nodal singularities; and 
\emph{(semi-)stable} if so are its geometric fibers. A prestable curve with marked points defined over a field is \emph{split} if the irreducible components of its normalization are geometrically irreducible and smooth, and the preimages of the nodes in the normalization are defined over the ground field. A family of prestable curves with marked points is called split if all of its fibers are so; cf. \cite[\S~2.22]{dJ96}.  
Let $U \subset Z$ be an open subset, and $(\cC, \sigma_\bullet)$ a family of curves with marked points over $U$. Then a \emph{model} of $(\cC, \sigma_\bullet)$ over $Z$ is a family of curves with marked points over $Z$, whose restriction to $U$ is $(\cC, \sigma_\bullet)$.

\subsection{Toric varieties and parameterized curves}\label{sec:families of parcur}
Throughout the paper, $M$ and $N$ are the lattice of characters, respectively cocharacters, of a toric variety $X$. The monomial functions are denoted by $x^m$ for $m\in M$. 
A \emph{para\-meterized curve} in the toric variety $X$ is a smooth projective curve with marked points $(C,\sigma_\bullet)$ together with a map $f\: C \to X$ such that $f(C)$ does not intersect torus orbits of codimension greater than one, and the image of $C \setminus \left( \bigcup_i \sigma_i \right)$ under $f$ is contained in the dense orbit of $X$.

A \emph{family of parameterized curves} $f\colon \cC \rightarrow X$ consists of the following data: 
	\begin{enumerate}
	    \item a family of smooth marked curves $(\cC \to B, \sigma_{\bullet})$ over a base $B$, and
	    \item a map $f\colon \cC \rightarrow X$, such that for any geometric point $p \in B$ the restriction $\cC_p \to X$ is a parameterized curve in the above sense. 
	\end{enumerate}

\subsection{Tropicalization of parameterized curves} The canonical tropicalization of parameterized curves is well established; see e.g. \cite{Tyo12,CHT23}. We recall the  construction here for convenience.  
Let $f\: C \to X$ be a parameterized curve, $C^0 \to \Spec(K^0)$ a prestable model, and $\widetilde C$ the fiber of $C^0$ over the closed point of $\Spec(K^0)$.
The \emph{tropicalization} $\trop(C)$ of $C$ with respect to the model $C^0$ is a tropical curve and we need to specify its underlying weighted graph $\GG$, as well as a length function $\ell$. The former is the dual graph of the central fiber $\widetilde C$, i.e., the vertices of $\mathbb G$ correspond to irreducible components of $\widetilde C$, the edges to nodes, and the legs to marked points. Together with the natural incidence relations this gives a finite graph $\mathbb G$, and the order on the set of marked points gives the order on its set of legs. The weight of a vertex $v$ of $\mathbb G$ is defined to be the geometric genus of the corresponding component $\widetilde C_v$ of the reduction $\widetilde C$. Finally, in order to define the length function on an edge $e\in E(\GG)$, let $z\in \widetilde C$ be the node corresponding to $e$. Then \'etale locally at $z$, the total space of $C^0$ is given by $xy=\lambda$ with $\lambda\in K^{00}$. We set $\ell(e) = \nu(\lambda)$. Although $\lambda$ depends on the \'etale neighborhood, its valuation does not, and hence the length function is well-defined.

Next, we explain how to construct the parameterization $h\:\trop(C)\to N_\RR$. To this end, observe that for any $m \in M$ the pullback $f^*(x^m)$ of the monomial $x^m$ gives a non-zero rational function on $C^0$ because the preimage of the dense orbit of $X$ is dense in $C$. Thus, for any irreducible component $\widetilde C_v$ of $\widetilde C$, there is a $\lambda_m \in K^\times$ such that $\lambda_m f^*(x^m)$ is an invertible function at the generic point of $\widetilde C_v$. Here $\lambda_m$ is unique up to an element of valuation $0$. The function $h(v)$, which associates to $m \in M$ the valuation $\val(\lambda_m)$, is clearly linear, and hence $h(v)\in N_\RR$. The parameterization $h\:\trop(C)\to N_\RR$ then is the unique piecewise integral affine function with values $h(v)$ at the vertices of $\trop(C)$, whose slopes along the legs are given as follows: for any leg $l$ and $m\in M$ we set $\frac{\partial h}{\partial \vec l}(m)=-\mathrm{ord}_{\sigma_i} f^*(x^m)$, where $\sigma_i$ is the marked point corresponding to $l$. By \cite[Lemma 2.23]{Tyo12}, $h\:\trop(C)\to N_\RR$ then is a parameterized tropical curve, that is, defines a harmonic map. It is called the \emph{tropicalization} of $f\: C \to X$ with respect to the model $C^0$. 

Clearly, the tropical curve $\trop(C)$ is independent of the parameterization and depends only on $C^0$. If the family $C\to X$ is stable and $C^0$ is the stable model, then the corresponding tropicalization is called simply {\em the tropicalization} of $C$ (resp. of $f\: C\to X$).

\begin{rem}\label{rem:slopesoflegs}
Let $l$ be a leg corresponding to a marked point $\sigma_i$. If $f$ maps $\sigma_i$ to the boundary divisor $D$ of $X$, then, by definition of parameterized curves, $f(\sigma_i)$ belongs to a unique component of $D$, and in particular, to the smooth locus of $D$. It follows that the direction of the slope $\frac{\partial h}{\partial \vec l}$ is determined by the irreducible component of $D$ that contains the image of $\sigma_i$. Furthermore, the integral length of $\frac{\partial h}{\partial \vec l}$ equals the multiplicity of $\sigma_i$ in $f^*D$.
Note also that $\frac{\partial h}{\partial \vec l} = 0$ if and only if $f(\sigma_i)$ is contained in the dense orbit of $X$.
\end{rem}

\subsection{Strictly semistable pairs and their tropicalization}
To tropicalize families, we will first tropicalize the base of the family. To do so, we shall use higher-dimensional tropicalizations. As mentioned in Remark~\ref{rem:tropfamother}, different versions of such tropicalizations have appeared in the literature; see, for example, \cite{MN15, ACGS, CCUW, Ran22, CGM}. Here we will follow the construction by Gubler, Rabinoff, and Werner \cite{GRW16}. Let us recall their construction.
\begin{defn}\label{defn:strictly semistable pairs}(\cite[Definition 3.1]{GRW16})
A \textit{strictly semistable pair} $(B^0,H^0)$ over $K^0$ consists of an irreducible proper flat scheme $B^0$ over $K^0$ and a sum $H^0=\sum H_k$ of distinct effective Cartier divisors $H_k$ on $B^0$ such that $B^0$ is covered by open subsets $U$ which admit an \'etale morphism 
\begin{equation}\label{eq:ssp}
\varphi\colon U\rightarrow \mathrm{Spec}(K^0[x_0,\dotsc,x_d]/\langle x_0\cdot\ldots\cdot x_a-\lambda\rangle)
\end{equation}
for some $a\leq d$ and $\lambda\in K^{00}$. Moreover, we assume that each $H_k$ has irreducible support and $H_k\cap U$ is either empty or defined by $\varphi^*(x_j)$ for some $a+1\leq j\leq d$. 
Similarly, define a \textit{strictly semistable pair} over $F^0$ by replacing $K^0$ with $F^0$. 
\end{defn}


For a strictly semistable pair $(B^0,H^0)$, denote by $\widetilde B$ the special fiber of $B^0$ with irreducible components $\widetilde B_j$. Set 
\begin{equation}\label{eq:divisor}
    D\coloneqq\widetilde B\cup\left(\bigcup_k H_k\right) \subset B^0 
\end{equation}
with irreducible components $\{D_i\}_{i \in I} = \{H_k\}_k \cup \{\widetilde B_j\}_j$. Then the $\widetilde B_j$'s are called the \emph{vertical components} of $D$ and the $H_k$'s the \emph{horizontal components}. The divisor $D$ admits a stratification by the irreducible components of $\bigcap_{i\in J}D_i\backslash\big(\bigcup_{i\not\in J}D_i\big)$ for any $J\subseteq I$. There is a natural partial order on the set of strata of $D$ given by $S\leq T$ if and only if $S\subseteq\overline T$. Let $\mathrm{Str}(B^0,H^0)$ denote the set of strata with vertical support, namely, those contained in $\widetilde B$. Then for any $S\in\mathrm{Str}(B^0,H^0)$, the closure $\overline S$ is smooth over $\widetilde K$, cf. \cite[\S 3.15]{GRW16}. In particular, each component of the reduction $\widetilde B_j$ is smooth over $\widetilde K$. Notice that our notation differs slightly from the one in {\em loc. cit.}, where the set of strata $\mathrm{Str}(B^0,H^0)$ is denoted by ${\rm str}(\widetilde B, H^0)$. 

We can now describe the {\em skeleton} of a strictly semistable pair $(B^0,H^0)$ following \cite[\S 4]{GRW16}. Notice, however, that in {\em loc. cit.} the construction is described for a formal strictly semistable pair obtained from $(B^0,H^0)$ by a formal completion with respect to a non-zero element of $K^{00}$. To a stratum $S\in \mathrm{Str}(B^0,H^0)$, one associates a polyhedron $\Delta_S$ as follows. Pick a sufficiently small open neighborhood $U$ of a general point of $S$ such that there exists an \'etale morphism $\varphi$ as in (\ref{eq:ssp}) and such that for any stratum $T$ the following holds: $S\le T$ if and only if $T\cap U\ne\emptyset$. Let $\{H_i\}_{1\leq i\leq b}$ be the horizontal divisors that intersect $U$, where $0\leq b\leq d-a$. We may assume that $H_i$ is defined by $\varphi^*(x_{a+i})$, and that $S$ is given by $\varphi^*(x_0)=\cdots=\varphi^*(x_{a+b})=0$. Set $\Delta_S:=\Delta(a,\lambda)\times \mathbb R^b_{\geq 0}$, where 
$$\Delta(a,\lambda)=\{y=(y_0,\dotsc,y_a)\in\mathbb R_{\geq 0}^{a+1}|y_0+\cdots +y_a=\nu(\lambda)\},$$
and notice that it is strictly convex, i.e., contains no lines. Following \cite{GRW16}, we call $\nu(\lambda)$ the {\em length} of $S$. 

For $U$ and $\Delta_S$ as above, consider the locus $U_S\subset U\setminus D$ of $K$-points $p$, whose specialization belongs to the union of strata $\bigcup_{S\le T} T$. There is a natural {\em tropicalization} map $\tau_U\colon U_S\rightarrow \Delta_S$ given by 
$$p\mapsto \left(\nu\left(\varphi^*(x_0)(p)\right),\dotsc,\nu\left(\varphi^*(x_{a+b})(p)\right)\right).$$ 
Then $\Delta_S$ is independent of the choice of $U$ and $\varphi$, up to reordering the coordinates, and the tropicalization map is also independent of the choice of $\varphi$, and compatible on the intersections for different choices of $U$; cf. \cite[\S 4.3]{GRW16}. 

If $S\le T$ are strata, then up to reordering of the coordinates, there exists $a'\leq a$ and $b'\leq b$, such that $T$ is given locally by $\varphi^*(x_0)=\cdots=\varphi^*(x_{a'})=\varphi^*(x_{a+1})=\cdots=\varphi^*(x_{a+b'})=0.$ Therefore, $\Delta_T$ can be identified naturally with a face of $\Delta_S$. Moreover, the tropicalization maps agree on $U_T$. Since $B^0$ is proper over $\Spec(K^0)$, any $K$-point specializes to one of the strata in $\mathrm{Str}(B^0,H^0)$, and therefore, the tropicalization maps glue to a well-defined tropicalization map $\tau$ from the set of $K$-points of $B^0\setminus D$ to the limit $S(B^0,H^0):=\bigcup_{S\in \mathrm{Str}(B^0,H^0)}\Delta_S$; cf. \cite[\S4.6, \S4.9]{GRW16}. In the sequel, e.g., in Theorem~\ref{thm:main thm}, we will abuse the notation and denote the tropicalization map $\tau$ by $\trop$.

\begin{defn}\label{defn:skeleton}(\cite[\S~4.6]{GRW16})
    The polyhedral complex  
    $S(B^0,H^0):=\bigcup_{S\in \mathrm{Str}(B^0,H^0)}\Delta_S$ is called the {\em skeleton} of the strictly semistable pair $(B^0,H^0)$.
\end{defn}

\begin{rem}
    The choice of the term ``skeleton'' is due to the fact that the polyhedral complex $S(B^0,H^0)$ embeds naturally into the Berkovich analytification $B^{\rm an}$ of $B$, and its image is a strong deformation retract of $B^{\rm an}$.
\end{rem}
The skeleton is naturally a rational polyhedral complex with an integral structure. Indeed, any face $\Delta_S$ admits a natural embedding $\Delta_S=\Delta(a,\lambda)\times \mathbb R^b_{\geq 0}\rightarrow \RR^{a+1+b}$, which induces the integral structure on $\Delta_S$. By construction, these integral structures are compatible for different faces of the skeleton.

\section{Main theorem}
In this section, we fix a pair of dual lattices $M$ and $N$, and a toric variety $X$ with lattice of monomials $M$.
Let $\overline{\CM}_{g,n}$ be the algebraic stack over $\mathbb Z$ classifying stable $n$-pointed curves of genus $g$.
Then $\overline{\CM}_{g,n}$ admits a finite surjective morphism from a projective scheme $\overline{M}$, over which there is a universal family of curves, cf. \cite[\S2.24]{dJ96}.

\begin{thm}\label{thm:main thm}
Let $(B^0,H^0)$ be a strictly semistable pair over $K^0$ and $(\cC^0\rightarrow B^0,\sigma^0_\bullet)$ a split family of stable marked curves over $B^0$. Let $B$, $H$ and $\cC$ be the generic fibers of $B^0$, $H^0$, and $\cC^0$, respectively. Suppose $\cC^0$ is obtained as a pullback of the universal family over $\overline M$ and is smooth over $B':=B\backslash H$. Suppose furthermore that $f \colon \cC|_{B'} \rightarrow X$ is a family of parameterized curves over $B'$. Then there exists a family of parameterized tropical curves $h\: \Gamma_\Lambda\to N_\RR$ over $\Lambda:=S(B^0,H^0)$ such that 
\begin{enumerate}
\item For any $\eta \in B'(K)$, the fiber of $(\Gamma_\Lambda, h)$ over $\trop(\eta)$ is the tropicalization of $f\:\cC_\eta\to X$;
\item For any face $W\in\CF(\Lambda)$ of codimension one (resp. of codimension at least two) such that the tropical curves over the interior of $W$ are weightless and $3$-valent except for at most one $4$-valent vertex, 
the induced map $\alpha\:\Lambda\to M^{\trop}_{g,n,\nabla}$ is either harmonic (resp. quasi-harmonic) or locally combinatorially surjective along $W$.
\end{enumerate}
\end{thm}


\begin{rem}\label{rem:ref}
    Property (1) can be considered as the ``defining property'' of the tropicalization of a family. Indeed, any natural tropicalization procedure should commute with point-wise tropicalization. In fact, in the case of families of curves, the existence of tropicalizations satisfying property (1) follows from the (more general) logarithmic tropicalization construction of Cavalieri, Chan, Ulirsch, and Wise, \cite[\S~7.1-7.3]{CCUW}. Constructions of tropicalizations of families of stable maps satisfying (1) can also be found in the literature but under the assumption that the characteristic of the ground field is zero, see, e.g., \cite{ACGS} or \cite{R17, Ran22}. 
    
    The main new ingredient in Theorem~\ref{thm:main thm} is the balancing properties of the induced map $\alpha$ described in (2). These balancing properties are crucial in some applications. For example, they can be used to prove realizability of parameterized curves in given families, and to prove irreducibility results for moduli of parameterized curves. In particular, the following corollary is used in the proof of irreducibility of Hurwitz schemes in any characteristic, \cite{CHT24b}.
\end{rem}

\begin{cor} \label{cor:main}
Let $X$ be a toric variety and $f\colon \cC \rightarrow X$ a family of parameterized curves defined over the field $K$. Assume that the base $B'$ of the family is quasi-projective, and consider the set of tropicalizations
$$\Sigma:=\left\{\left[ \trop(f_\eta, C_\eta)\right]\; :\; \eta\in B'(K) \right\}\subseteq M^{\trop}_{g,n,\nabla}.$$
Let $M_{[\Theta]}\subset M^{\trop}_{g,n,\nabla}$ be a stratum. Then,
\begin{enumerate}
\item The closure $\overline\Sigma\subseteq M^{\trop}_{g,n,\nabla}$ is the image of a rational polyhedral complex $\Lambda$ under a piecewise integral affine map;
\item If $M_{[\Theta]}$ is weightless and $3$-valent and $\dim(\Sigma\cap M_{[\Theta]})=\dim M_{[\Theta]}$, then $\overline M_{[\Theta]}\subseteq \overline{\Sigma}$;
\item If $M_{[\Theta]}$ is weightless and almost $3$-valent and $\overline{\Sigma}$ contains one of its adjoint weightless and $3$-valent strata, then $\overline{\Sigma}$ contains all other such strata as well. 
\end{enumerate}
\end{cor}

\begin{proof}
After replacing $B'$ with a finite covering, we may assume that the family of marked curves over $B'$ is a pullback of the universal family over the projective scheme $\overline M$. Plainly, it suffices to prove the corollary for each irreducible component of $B'$. Furthermore, we may replace $B'$ with $B'_{\rm red}$, and hence we will assume that $B'$ is integral, i.e., a quasi-projective variety. Thus, there is a projective variety $B$ containing $B'$ as a dense open subset. After replacing $B$ with the closure of the graph $B'\to \overline M \times B$, we may assume that the family of marked curves over $B'$ extends to $B$, and the latter family is also a pullback of the universal family over $\overline M$. Set $H:=B\setminus B'$. 

By \cite[\href{https://stacks.math.columbia.edu/tag/0C2S}{Tag~0C2S}]{stacks-project}, $B$ admits a flat projective model $B^0$ over $K^0$. Let $\overline M^0$ be the trivial model of $\overline M$ over $K^0$. After replacing $B^0$ with the schematic image of the graph $B\to \overline M^0 \times_{\Spec(K^0)} B^0$, which is also flat over $K^0$ by {\em loc. cit.}, we may assume that the family of marked curves extends to $B^0$. By \cite[Lemma~5.3]{dJ96}, there exist a projective alteration $\varphi_1 \colon B_1^0 \to B^0$ such that the pullback of $\cC$ to $B_1^0$ is split over $B_1^0$. Replacing $B_1^0$ with an irreducible component that dominates $B^0$ if necessary, we may assume that $B_1^0$ is integral. Since the loci $\Sigma$ associated to $B'$ and to $\varphi_1^{-1}(B')$ coincide, we may assume that $B_1^0=B^0$ and hence the family $\cC$ is split over $B^0$.
Finally, by \cite[Theorem~6.5]{dJ96}, there exists an alteration $\varphi_2 \colon B_2^0 \to B^0$ such that $B_2^0$ is projective over $K^0$, the generic fiber of $B_2^0\to\Spec(K^0)$ is irreducible, and the pair $\left(B_2^0, \overline{\varphi_2^{-1}(H)}_{\rm red}\right)$ is strictly semistable. Here $\overline{\varphi_2^{-1}(H)}_{\rm red}$ is the closure of $\varphi_2^{-1}(H)_{\rm red}$ in $B_2^0$.
As before, we may assume that $B_2^0=B^0$. 

Denote by $H^0$ the union of the horizontal components of $H\cup \tilde{B}$. Then $(B^0, H^0)$ is a strictly semistable pair, and the family of marked curves extends to $B^0$ by pulling back the universal family over $\overline M$. Therefore, we may assume that we are given all the data $B^0, H^0, \cC^0, \sigma^0_\bullet, B'$ satisfying the assumptions of Theorem~\ref{thm:main thm}. Consider the tropicalization $h\: \Gamma_\Lambda\to N_\RR$ and the induced map $\alpha\:\Lambda\to M^{\trop}_{g,n,\nabla}$ provided by Theorem~\ref{thm:main thm}.

(1) Since $\alpha$ is piecewise integral affine, it follows that $\alpha(\Lambda)\subseteq M^{\trop}_{g,n,\nabla}$ is closed. Now, by Theorem~\ref{thm:main thm} (1), $\Sigma\subset \alpha(\Lambda)$, and therefore $\overline\Sigma\subset \alpha(\Lambda)$. Furthermore, $\alpha^{-1}(\Sigma)$ is dense in $\Lambda$. Thus, $\Sigma$ is dense in $\alpha(\Lambda)$, and hence $\overline\Sigma=\alpha(\Lambda)$ as asserted.

(2) Set $k:=\dim M_{[\Theta]}$, and assume to the contrary that $\overline M_{[\Theta]}\nsubseteq \alpha(\Lambda)$. By (1), $\alpha(\Lambda)\subseteq M^{\trop}_{g,n,\nabla}$ is closed, and therefore $M_{[\Theta]}\nsubseteq \alpha(\Lambda)$. Furthermore, there is a face $W\in \CF(\Lambda)$ such that $\alpha(W^\circ)$ has dimension $k-1$, $\alpha\left(\Star(W^\circ)\right)$ has dimension $k$, and $\alpha(W^\circ)$ belongs to the relative boundary of $\alpha(\Lambda)$ in $M_{[\Theta]}$.

After replacing $W$ with one of its faces, we may assume that $\dim(W)=k-1$. Indeed, since $\dim(\alpha(W^\circ))=k-1$, the dimension of $W$ is at least $k-1$. If it is greater than $k-1$, then the kernel of the differential $d\alpha$ along $W$ is non-trivial. But $\Lambda$ is the skeleton of a strictly semistable pair, and therefore $W$ contains no lines. Thus, $\alpha(W)=\alpha(\partial W)$, and by induction, we can find a face of $W$ of dimension $k-1$ satisfying the requirements.

By construction, $\alpha\left(\Star(W^\circ)\right)$ belongs to a half-space supported by $\alpha(W^\circ)$ and the differential $d\alpha$ is not identically zero on $\Star(W^\circ)$. Therefore $\alpha$ is not quasi-harmonic along $W$ contradicting Theorem~\ref{thm:main thm} (2).

(3) Let $M_{[\Theta']}$ be a weightless and $3$-valent stratum adjacent to $M_{[\Theta]}$ that is contained in $\overline\Sigma$, and $M_{[\Theta'']}$ be another weightless and $3$-valent stratum adjacent to $M_{[\Theta]}$. Then there exists a face $W'\in\mathcal F(\Lambda)$ and a face $W$ of $W'$ of codimension one such that $\alpha\left((W')^\circ\right)\subseteq M_{[\Theta']}$ has dimension $\dim M_{[\Theta']}$, and $\alpha(W^\circ)\subseteq M_{[\Theta]}$. Thus, $\alpha$ is locally combinatorially surjective along $W$ by Theorem~\ref{thm:main thm} (3). In particular, there exists a face $W''\in\mathcal F(\Lambda)$ containing $W$ as a codimension-one face such that $\alpha\left((W'')^\circ\right)\subseteq M_{[\Theta'']}$, and hence 
$$\dim \alpha\left(W''\right) \ge \dim \alpha(W)+1=\dim\left(M_{[\Theta'']}\right)$$
since $\alpha|_{W''}$ is an affine map and $\alpha(W)\subsetneq\alpha(W'')$.
Therefore, $\overline M_{[\Theta'']}\subseteq \overline{\Sigma}$ by (2), and we are done.
\end{proof}

\begin{rem}
    The above results, while applicable to the more general case of weightless and (almost) $3$-valent strata, are conceived with a view towards their application in \cite{CHT24b}. There we will use simple walls and nice strata of \cite{CHT23}, which are weightless and (almost) $3$-valent strata, which in addition are regular, i.e., of the expected dimension. Two weightless and (almost) $3$-valent, but not regular, strata were studied already in \cite{CHT22}:

    A non-regular, weightless and almost $3$-valent type $\Theta$ is given by a contracted loop adjacent to the $4$-valent vertex, and no other contracted edges. In this case, the stratum $M_{[\Theta]}$ has dimension one more than expected. While Corollary~\ref{cor:main} (2) does not apply in this case, since the stratum is not $3$-valent, the claim still holds: since $\alpha$ is not locally combinatorially surjective at $M_{[\Theta]}$ by the proof of \cite[Lemma~4.6]{CHT22}, it is (quasi-)harmonic by Theorem~\ref{thm:main thm}. This in turn is the only statement we used to deduce Corollary~\ref{cor:main} (2).  

    A non-regular, weightless and $3$-valent type $\Theta$ on the other hand appears when $\Theta$ is $3$-valent but contains a single flattened cycle as in \cite[Definition~3.7]{CHT22}. In this case, we again have dimension one more than expected, and $M_{[\Theta]}$ contains curves that are not realizable \cite[Proposition~3.8]{CHT22}. In particular, the condition of local surjectivity of $\alpha$ in Corollary~\ref{cor:main} (2) can never be satisfied. However, one can describe the realizable sublocus in $M_{[\Theta]}$ explicitly, and (quasi-)harmonicity of $\alpha$
    again ensures that if $\alpha$ is locally surjective at a point in this locus, it is surjective onto the whole realizable locus in $M_{[\Theta]}$. 
\end{rem}

\medskip

The rest of the section is devoted to the proof of Theorem~\ref{thm:main thm}. Since $K$ is the algebraic closure of a complete discretely valued field $F$, any $K$-scheme of finite type is defined over a finite extension $F'$ of $F$, which is also a complete discretely valued field since so is $F$. Thus, we may view a $K$-scheme of finite type as the base change of a scheme over a finite extension of $F$. In the proofs below, we will work over $F'$ in order to have a well-behaved total space of the families we consider. To ease the notation, we will assume that all points we are interested in are defined already over $F$ and integral models -- over $F^0$. For the rest of the proof, we fix a uniformizer $\pi\in F^{00}$, and assume without loss of generality that $\nu(\pi)=1$. In particular, we may assume that $\eta$ and $f$ are defined over $F$ and $(\cC^0\rightarrow B^0,\sigma^0_\bullet)$ over $F^0$. 

Throughout the proof, we will use the following notation: $\widetilde B$ and $\widetilde \cC$ will denote the special fibers, 
$s:=\red(\eta)\in \widetilde B$ the reduction of $\eta$, and $\widetilde\cC_s$ the corresponding fiber of $\cC^0\to B^0$. We set $D_{B^0}:=\widetilde B \cup H$ and $D_{\cC^0}:=\widetilde \cC \cup  \cC_H \cup \left( \bigcup_j \sigma_j^0 \right)$, where $\cC_H$ is the restriction of $\cC$ to $H$. Then, the pullback of any monomial function $f^*(x^m)$ is regular and invertible on $\cC^0 \setminus D_{\cC^0}$ by the assumptions of the theorem. Recall that $D_{B^0}$ has a stratification, and we denote by $\mathrm{Str}(B^0,H^0)$ the set of strata with vertical support. 

Let $S\in \mathrm{Str}(B^0,H^0)$ be the stratum containing $s$, and $s\in U\subset B^0$ an open neighborhood admitting an \'etale morphism
$\varphi\colon U\rightarrow \mathrm{Spec}\left(F^0[x_0,\dotsc,x_d]/\langle x_0\cdots x_a-\lambda\rangle\right)$ as in Definition~\ref{defn:strictly semistable pairs}, such that $S\cap U$ is given by
$\varphi^*(x_0)=\cdots=\varphi^*(x_{a+b})=0$ for some $0\le b\leq d-a$. Then the tropicalization $$\trop(\eta)=\left(\nu\left(\varphi^*(x_0)(\eta)\right),\dotsc,\nu\left(\varphi^*(x_{a+b})(\eta)\right)\right)$$
is contained in the interior of $\Delta_{S}=\Delta(a,\lambda)\times \mathbb R^b_{\geq 0}$. Without loss of generality we may assume that $\lambda=\pi^l$, where $l=\nu(\lambda)$ is the length of $S$.
We denote by $\widetilde\cC_S$ and $\widetilde\cC_{\overline S}$ the restriction of $\widetilde \cC$ to $S$ and $\overline S$, respectively. Similarly, for an irreducible component $D_i\subseteq D_{B^0}$, given locally by $\varphi^*(x_i)=0$ for some $0\leq i\leq a+b$, we denote by $\cC_{D_i}$ the restriction of $\cC^0$ to  $D_i$. Finally, we denote by $\psi$ the map from the preimage of $U$ in $\cC^0$ to $\mathrm{Spec}\left(F^0[x_0,\dotsc,x_d]/\langle x_0\cdots x_a-\pi^l\rangle\right)$.

\begin{rem}
    As already mentioned in Remark~\ref{rem:ref}, the existence of tropicalizations of families of (parameterized) curves satisfying property (1) is known to the experts, see e.g., \cite{ACGS,CCUW,R17,Ran22}. However, the results in {\em loc. cit.} are stated in a different language and some are proven under the assumption that the characteristic is zero. Therefore, we include a proof of the existence of tropicalizations satisfying (1) for the convenience of the reader.
\end{rem}

\subsection{Proof of part (1) Theorem~\ref{thm:main thm}}
We need to specify a datum (\dag) as in Section~\ref{subsubsec:family of parameterized tropical curves} that satisfies the compatibility conditions of Definition~\ref{defn:family of tropical curves}. Since the tropicalizations of $K$-points give rise only to rational points of $\Lambda$, we will work with those and extend the family by linearity to the non-rational points in the very end. To ease the notation, for a face $W=\Delta_T\in \CF(\Lambda)$ corresponding to a stratum $T \in \mathrm{Str}(B^0, H^0)$, set $\Theta_T:=\Theta_W$, $\GG_T:=\GG_W$, $\ell_T:=\ell_W$, and $h_T:=h_W$. Similarly, if $W'=\Delta_{T'}\in \CF(\Lambda)$ is such that $W$ is a face of $W'$, i.e., $T'\le T$, then set $\phi_{T',T}:=\phi_{W',W}$.

\subsubsection*{Step 1: The tropicalization $\trop(\cC_\eta)$ depends only on $\trop(\eta)\in \Lambda$, the underlying graph of $\trop(\cC_\eta)$ depends only on $S$, and the length of each edge of $\trop(\cC_\eta)$ is an integral affine function over $\Delta_S^\circ$ which extends to $\Delta_S$ in a compatible way.}
  
By definition, the underlying graph of $\trop(\cC_\eta)$ is the dual graph $\mathbb G_s$ of $\widetilde\cC_s$. 
We first show that $\mathbb G_s$ only depends on the stratum $S$. Let $z$ be a node of $\widetilde \cC_s$. Since $\cC^0$ is pulled back from $\overline M$, there exist an \'etale neighborhood $V$ of $s$ in $B^0$ and a function $g_z\in \mathcal O_V(V)$ vanishing at $s$ such that the family $\cC^0\times_{B^0}V$ is given by $xy=g_z$ \'etale locally near $z$. After shrinking $V$, we may assume that the latter is true in a neighborhood of any node of $\widetilde\cC_s$, and that the morphism $V\to B^0$ factors through $U\subseteq B^0$. By assumption, $\cC$ is smooth over $B\backslash H$, hence each $g_z$ is invertible on the complement of $D_{B^0}$. However, $V$ is normal, and $g_z(s)=0$ for all nodes $z$, hence each $g_z$ must vanish along some component of $D_{B^0}\cap V$ that intersects $S$, and therefore contains $V\cap S$. Thus, all the $g_z$'s vanish identically along $S$. This implies that the dual graph $\mathbb G_s$ is \'etale locally constant along $S$. Moreover, since $\widetilde \cC_S\rightarrow S$ is split, the identification of $\mathbb G_s$ with the dual graph of the fiber over the generic point of $S$ is canonical. Therefore, $\mathbb G_s$ only depends on $S$. Set $\GG_S:=\GG_s$. Then for any $S\le T\in \mathrm{Str}(B^0,H^0)$ there is a natural contraction $\phi_{S,T}\colon \GG_{S}\rightarrow \GG_{T}$ obtained by the specialization of the curve over the generic point of $T$ to the curve over the generic point of $S$.

Pick a node $z\in \widetilde{\cC}_s$. We claim that there exist integers $n_i$ and $n_\pi$ such that 
\begin{equation}\label{eq:monomialfunction}
g_z\cdot\pi^{-n_\pi}\cdot\prod_{i=0}^{a+b} \varphi^*(x_i)^{-n_i}
\end{equation}
is regular and invertible along $S$. Here, by abuse of notation, we denote by $\varphi^*(x_i)$ the pullback of the monomial $x_i$ to $V$ rather than $U$. Consider the log-structure on $\Spec(F^0)$ given by the uniformizer $\pi$, the {\em monomial} log-structure on $\mathrm{Spec}(F^0[x_0,\dotsc,x_d]/\langle x_0\cdot\ldots\cdot x_a-\pi^l\rangle)$ given by $\pi$ and the monomials $x_i$'s, and the monomial log-structure on $V$ given by the pullback of the monomial log-structure on $\mathrm{Spec}(F^0[x_0,\dotsc,x_d]/\langle x_0\cdot\ldots\cdot x_a-\pi^l\rangle)$. Then $$\mathrm{Spec}(F^0[x_0,\dotsc,x_d]/\langle x_0\cdot\ldots\cdot x_a-\pi^l\rangle)\to \Spec(F^0)$$ is log-smooth by \cite[\S 8.1]{Kato94} and $V\to \mathrm{Spec}(F^0[x_0,\dotsc,x_d]/\langle x_0\cdot\ldots\cdot x_a-\pi^l\rangle)$ log-\'etale by \cite[Proposition~3.8]{Kato89}. Since $\Spec(F^0)$ is log-regular, so is $V$ by \cite[\S 8.2]{Kato94}. Now, since $V$ is log-regular, the monomial log-structure is given by the pushforward of the sheaf of invertible functions on the complement of $D_{B^0}$ by \cite[Theorem~11.6]{Kato94}. Finally, since $g_z$ is invertible away from $D_{B^0}$, it follows that the function $g_z$ is a monomial in $\pi$ and $\varphi^*(x_i)$'s up-to an invertible function as asserted. Proofs of similar claims in slightly different settings can be found in \cite[Lemma~4.8.2]{Wlo22} and \cite[Proposition~2.11]{Gubler07}.

Next, we compute the lengths of the edges of $\GG_s$ in $\trop(\cC_\eta)$. By \eqref{eq:monomialfunction}, the length of $\gamma\in E(\mathbb G_s)$ corresponding to $z$ is given by 
$$
\ell(\gamma)=\nu(g_z(\eta))=n_\pi+\sum_{i=0}^{a+b}n_i\cdot \nu\left(\varphi^*(x_i)(\eta)\right)=\sum_{i=0}^{a+b}k_i\cdot \nu\left(\varphi^*(x_i)(\eta)\right),
$$
where $k_i:=n_i+\frac{1}{l}n_\pi=\frac{1}{l}\mathrm{ord}_{D_i}(g_z)$ if $i\le a$ and $k_i:=n_i=\mathrm{ord}_{D_i}(g_z)$ otherwise. We claim that the $k_i$'s are independent of the choice of $s\in S$. Indeed, for another $s'\in S$, the corresponding $g_z$ and $g'_z$ differ by an invertible function on the intersection of the corresponding \'etale neighborhoods of $s$ and $s'$, and therefore also at the generic points of the $D_i$'s for all $0\leq i\leq a+b$. Thus, $k_i=\mathrm{ord}_{D_i}(g_z)=\mathrm{ord}_{D_i}(g'_z)=k'_i$ as claimed. Furthermore, we see that the length function $\ell(\gamma)$ depends only on the tropicalization $q:=\trop(\eta)=\left(\nu\left(\varphi^*(x_0)(\eta)\right),\dotsc, \nu\left(\varphi^*(x_{a+b})(\eta)\right)\right)$ rather than the point $\eta$ itself.
 
Set $y_i:=\nu\left(\varphi^*(x_i)(\eta)\right)$. Then, $q=(y_0,\dotsc,y_{a+b})$ and the length function
$$\ell_S(\gamma,q)=n_\pi+\sum_{i=0}^{a+b} n_iy_i=\sum_{i=0}^{a+b} k_iy_i$$ 
is an integral affine function on $\Delta_{S}^\circ$ that extends naturally to $\Delta_S$. It remains to show that $\ell_S(\gamma,q)$ satisfies the compatibility condition of Definition~\ref{defn:family of tropical curves} (2). Note that $g_z$ is defined over an \'etale open subset of $B^0$ which intersects all strata $T\in\mathrm{Str}(B^0,H^0)$ satisfying $S\le T$. Let $T$ be the stratum given locally by 
$$\varphi^*(x_0)=\cdots =\varphi^*(x_{a'})=0=\varphi^*(x_{a+1})=\cdots=\varphi^*(x_{a+b'}),$$
for some $a'\leq a$ and $b'\leq b$. Then $\Delta_{T}$ is the face of $\Delta_S$ given by $$y_{a'+1}=\cdots=y_a=y_{a+b'+1}=\cdots=y_{a+b}=0.$$

If the node $z$ is smoothed out in the family of curves over $T$, then $g_z$ is invertible at the generic point of $D_i$ with $0\leq i\leq a'$ or $a+1\leq i\leq a+b'$, and the edge $\gamma$ gets contracted to a vertex in $\GG_{T}$. Hence $k_i=0$ for $0\leq i\leq a'$ or $a+1\leq i\leq a+b'$ and the compatibility holds:
$$\ell_{T}(\phi_{S,T}(\gamma),q)=0=\ell_S(\gamma,q),\ \mathrm{for\ all}\ q\in \Delta_{T}.$$

Assume now that $z$ is the specialization of a node of the curve over the generic point of $T$. Then $\gamma$ is not contracted by $\phi_{S,T}$. In this case, again for any $s'\in T$, the corresponding $g'_z$ and $g_z$ differ by a function invertible at the generic points of all $D_i$ that contain $T$, namely, with $0\leq i\leq a'$ or $a+1\leq i\leq a+b'$. Therefore, $\mathrm{ord}_{D_i}(g_z)=\mathrm{ord}_{D_i}(g'_z)$ for those $D_i$'s, and hence the restriction of $\ell_S(\gamma,\cdot)$ to $\Delta_T$ coincides with $\ell_T(\phi_{S,T}(\gamma),\cdot)$ as needed.

\subsubsection*{Step 2: The parameterization $h_\eta\colon\trop(\cC_\eta)\rightarrow N_\RR$ depends only on $\trop(\eta)$, and varies integral affine linearly on $\Delta_S^\circ$ with constant combinatorial type. It also extends to $\Delta_S$ in a compatible way.}\ 

Since each section $\sigma^0_\bullet$ of the family $\cC\rightarrow B^0$ is mapped to a unique toric divisor or the dense orbit in $X$, the slopes of the legs of $\trop(\cC_\eta)$ are independent of $\eta$ by Remark~\ref{rem:slopesoflegs}. Therefore, the extended degree of $h_\eta\colon\trop(\cC_\eta)\rightarrow N_\RR$ is independent of $\eta$ and will be denoted by $\overline\nabla$.

Recall that the closure $\overline S$ of $S$ is a component of the intersection of $D_0,\dotsc,D_{a+b}$.
Let $u$ be a vertex of $\GG_{S}$. 
Denote by
$\widetilde \cC_u$ the component of $\widetilde \cC_S$ corresponding to $u$ and by $\widetilde\cC_{s,u}$ the component of $\widetilde \cC_s$ corresponding to $u$.
As in Step 1, let $q=(y_0,\dotsc,y_{a+b}):=\trop(\eta)$.
Let $\widetilde \cC_u'\subset\widetilde \cC_u$ be the open subset of non-special points. Then in a neighborhood of $\widetilde \cC_u'$, the toroidal structure on $\mathscr C^0$ is pulled back from the one on $B^0$.
Therefore, as in Step 1, for any $m\in M$, there exist integers $n_i$ for $0\leq i\leq a+b$ and $n_\pi$ such that  $$f^*(x^m)\cdot\pi^{-n_\pi}\cdot\prod_{0\leq i\leq a+b} \psi^*(x_i)^{-n_i}$$ is regular and invertible 
along $\widetilde \cC_u'$. 
Hence, by definition, the inner product of $h_\eta(u)\in N_\RR$ with $m$ is given by $$-\mathrm{ord}_{\widetilde \cC_{s,u}}\left(f^*(x^m)|_{\cC_{\overline{\eta}}}\right)=-n_\pi-\sum_{i=0}^{a+b}n_i\cdot \nu\left(\varphi^*(x_i)(\eta)\right)=-n_\pi-\sum_{i=0}^{a+b}n_iy_i,$$ which only depends on $q=\trop(\eta).$
Let us write $h_S(u,q)=h_\eta(u)$, then $h_S(u,\cdot)$ is an integral affine function on $\Delta_S$.
By a similar argument as in Step 1, we see that this function satisfies the compatibility of Definition~\ref{defn:family of tropical curves} (3).

As the parameterization $h_\eta$ varies continuously with respect to $q=\trop(\eta)$, while the set of combinatorial types with underlying graph $\GG_S$ is a discrete set, we see that the combinatorial type of $h_\eta$ is constant, which we denote by $\Theta_S$.

\subsubsection*{Step 3: The conclusion of the proof of (1).}
In the first two steps, we have constructed the data of a family of parameterized tropical curve over $\Lambda$ - the extended degree, the combinatorial type for each face of $\Lambda$, the weighted contractions, and the integral affine functions $\ell_S$ and $h_S$. Furthermore, we have verified the compatibility conditions of Definition~\ref{defn:family of tropical curves}, and have seen that the resulting family of parameterized tropical curves over $\Lambda$ satisfies assertion (1).

\subsection{Proof of Part (2) of Theorem~\ref{thm:main thm}} 
Denote by $\chi\colon B^0\rightarrow \overline M$ the map inducing the family $\cC^0$. Let $S\in \mathrm{Str}(B^0,H^0)$ be the stratum corresponding to $W$. Then the closure $\overline S$ of $S$
is a proper and smooth variety. Denote by $\Theta$ the combinatorial type of $\alpha(W^\circ)$. There are two cases to consider:

Case 1: \textit{$\chi$ contracts $\overline S$}. In this case we will prove that $\alpha$ is harmonic along $W$ if $W$ has codimension one and quasi-harmonic along $W$ otherwise.  Set $p:=\chi(\overline S)$ and let $C_p$ be the fiber of the universal family over $p$. Since $\cC^0$ is a pullback of the universal family along $\chi$, it follows that $\widetilde \cC_{\overline S}=\overline S\times C_p$. Furthermore, $\GG_S$ is the dual graph of $C_p$, and $\alpha$ maps $\Star(W^\circ)$ to $M_{[\Theta]}$, which lifts to a map to $M_\Theta\subset \RR^{|E(\GG_S)|} \times N_\RR^{|V(\GG_S)|}$. By abuse of notation, the lifting is also denoted by $\alpha$. 

Recall that each point $s\in S$ admits an open neighborhood $s\in U\subset B^0$ and an \'etale morphism
$\varphi\colon U\rightarrow \mathrm{Spec}\left(F^0[x_0,\dotsc,x_d]/\langle x_0\cdots x_a-\pi^l\rangle\right)$ such that $S\cap U$ is given by $\varphi^*(x_0)=\cdots=\varphi^*(x_{a+b})=0$ for some $0\le b\leq d-a$, and hence
$W=\Delta(a,\pi^l)\times \mathbb R^b_{\geq 0}$. Let $v_0,\dotsc,v_a$ be the vertices of $W$. For $a+1\leq i\leq a+b$, let $\vec e_i$ be the standard basis of $\mathbb R^b$ and set $v_i=v_0+\vec e_i$. Then $\Lin(\alpha(W))$ is generated by $\alpha(v_i)-\alpha(v_0)$ for $1\leq i\leq a+b$.

Let $\{S_i\}_{i=a+b+1}^{a+b+r+t}$ be the set of strata of codimension one in $\overline S$. Each $S_i$ corresponds to a face $W_i\in\mathcal F(\Lambda)$ containing $W$ as a codimension-one face. Suppose $\overline S_i=\overline S\cap D_i$ for some component $D_i$ of $D_{B^0}$. We may further assume that $D_i$ is vertical for $a+b+1\leq i\leq a+b+r$, and horizontal for $a+b+r+1\leq i\leq a+b+r+t$. Note that the $D_i$'s are uniquely determined by the $S_i$'s but may not necessarily be distinct. In the sequel, we will mainly consider the following components of $D_{B^0}$, which have non-empty intersection with $\overline S$: the components $D_0,\dotsc,D_{a+b}$, which are distinct and contain $\overline S$; and the components $D_{a+b+1},\dotsc,D_{a+b+r+t}$ constructed above.

For $a+b+1\leq i\leq a+b+r$,
let $l_i$ be the length of $S_i$ and  $v'_i$ be the vertex of $W_i$ not contained in $W$. In fact, $l_i=l$ unless $a=0$. Let $\vec e_i=\frac{1}{l_i}(v'_i-v_0)$ be the primitive integral vector parallel to $v'_i-v_0$, and set $v_i:=v_0+\vec e_i$. For $a+b+r+1\leq i\leq a+b+r+t$ we have $W_i=W\times \RR_{\geq 0}$. In this case, set $v_i:=v_0+\vec e_i$, where $\vec e_i$ is the unit normal vector to $W$ in $W_i$. Then $W_i$ is contained in $W+\RR_{\geq 0}\vec e_i$ for $a+b+1\leq i\leq a+b+r+t$. Consequently, we have $\Star(W)=\{\vec e_i+N_W\}_{a+b+1\leq i\leq a+b+r+t}$ and $\frac{\partial \alpha}{\partial\vec e_i}=\alpha(v_i)-\alpha(v_0)$. Thus, to prove quasi-harmonicity (resp. harmonicity), it remains to show that
\begin{equation}\label{eq:balancing in proof}
   \sum_{i=a+b+1}^{a+b+r+t}a_i\left(\alpha(v_i)-\alpha(v_0)\right)\in\Lin(\alpha(W))
\end{equation}
for some positive integers $a_i$ (resp. $a_i=1$).

Since $\overline S$ is \'etale locally isomorphic to an affine space and the $\overline S_i$'s correspond to the coordinate hyperplanes, we can pick a (possibly reducible) curve $C$ in $\overline S$ that has positive intersection numbers with each $\overline S_i$ and none of whose components is contained in $\bigcup \overline{S}_i$.  In particular, when $\dim S=1$ (equivalently, $W$ has codimension one in $\Lambda$), we set $C:=\overline S$. Set $l_i:=l$ for $i\le a$, and $l_i:=1$ for $a+1\le i\le a+b$ and for $i>a+b+r$. Then (\ref{eq:balancing in proof}) would follow from the following equality
\begin{equation}\label{eq:balance claim}
\sum_{i=a+b+1}^{a+b+r+t}\overline S_i\cdot C \ \big(\alpha(v_i)-\alpha(v_0)\big)=-\sum_{i=0}^{a+b}l_iD_i\cdot C\  \big(\alpha(v_i)-\alpha(v_0)\big),
\end{equation} 
where the intersection product on the left is in $\overline S$, while the one on the right is in $B^0$. Note that when $\dim S=1$ we have $\overline S_i\cdot C=1$ by construction.

We now prove \eqref{eq:balance claim} coordinatewise. Let $\gamma\in E(\GG_{S})$ be an edge corresponding to a node $z\in C_p$.   
Let $u,u'\in V(\GG_S)$ be the vertices adjacent to $\gamma$, let $C_u$ and $C_{u'}$ be the corresponding components of $C_p$, and $\widetilde{\cC}_u, \widetilde{\cC}_{u'}, Z$ the pullbacks of $C_u, C_{u'}, z$ to $\widetilde{\cC}_{\overline S}$, respectively. The universal curve over $\overline M$ is given \'etale locally at $z$ by $xy=m_z$, where $m_z$ is defined on an \'etale neighborhood of $p$, and vanishes at $p$. Thus, $\cC^0$ is given by $xy=g_Z:=\chi^*(m_z)$ in an \'etale neighborhood of $Z$. Notice that we constructed a function defined in a neighborhood of the whole family of nodes $Z$, which globalizes the local construction of Step 1 in the particular case we consider here.

According to Step 1, the $\gamma$-coordinate of $\alpha(v_i)-\alpha(v_0)$ is equal to $\frac{1}{l_i}\left(\ord_{D_i}(g_Z)-\ord_{D_0}(g_Z)\right)$ when $D_i$ is vertical, namely $0\leq i\leq a$ or $a+b+1\leq i\leq a+b+r$; and equal to $\ord_{D_i}(g_Z)$ otherwise. Set $k_i:=\frac{1}{l_i}\ord_{D_i}(g_Z)$.

Then the $\gamma$-coordinate of the left-hand side of \eqref{eq:balance claim} is given by 
$$\sum_{i=a+b+1}^{a+b+r+t}k_i\overline S_i\cdot C-\sum_{i=a+b+1}^{a+b+r}k_0\frac{l}{l_i}\overline S_i\cdot C;$$
while on the right-hand side it is given by 
$$-\sum_{i=0}^{a+b} l_ik_iD_i\cdot C+\sum_{i=0}^al_0k_0 D_i\cdot C=-\sum_{i=0}^{a+b} l_ik_iD_i\cdot C+\sum_{i=0}^alk_0 D_i\cdot C.$$
Therefore, \eqref{eq:balance claim} is equivalent to
\begin{equation}\label{eq:0=0}
    \sum_{i=0}^{a+b} l_ik_iD_i\cdot C+\sum_{i=a+b+1}^{a+b+r+t}k_i \overline S_i\cdot C=lk_0\Big(\sum_{i=0}^aD_i\cdot C+\sum_{i=a+b+1}^{a+b+r}\frac{1}{l_i}\overline S_i\cdot C\Big),
\end{equation}
which holds true since both sides vanish. Indeed, since $\overline S_i\cdot C=l_iD_i\cdot C$ for all $i>a+b$, the left-hand side of \eqref{eq:0=0} is $\mathrm{div}(g_Z)\cdot C$, which vanishes since $C$ is complete, while the right-hand side is $k_0l\widetilde B\cdot C=0$ since $\widetilde B$ is a principal divisor of $B^0$ defined by $\pi$.

Next we verify \eqref{eq:balance claim} for the $u$-coordinates. Pick $u\in V(\GG_{S})$. For $1\leq i\leq a+b+r+t$, let $\cC_i$ be the irreducible component of $\cC_{D_i}$ that contains some fiber of $\widetilde\cC_u$. Denote $k'_i=\frac{1}{l_i}\ord_{\cC_i}f^*(x^m)$. 
In this case, by taking inner product of both sides of \eqref{eq:balance claim} with an arbitrary $m\in M$, a similar computation as above shows that \eqref{eq:balance claim} is equivalent to 
$$\sum_{i=0}^{a+b} l_ik'_iD_i\cdot C+\sum_{i=a+b+1}^{a+b+r+t}k'_i \overline S_i\cdot C=lk'_0\Big(\sum_{i=0}^aD_i\cdot C+\sum_{i=a+b+1}^{a+b+r}\frac{1}{l_i}\overline S_i\cdot C\Big).$$
Take a section $\widetilde \sigma_S$ of $\widetilde \cC_u\rightarrow \overline S$
that is disjoint with any section $\sigma^0_\bullet$ of $\cC^0$ and the singular locus of $\cC^0$, and let $\widetilde \sigma_C$ (resp. $\widetilde \sigma_i$) be the restriction of $\widetilde \sigma_S$ on $C$ (resp. $\overline S_i$). 
By projection formula, we have $D_i\cdot C=\cC_i\cdot \widetilde \sigma_C$. Then \eqref{eq:balance claim} is equivalent to
$$\sum_{i=0}^{a+b} l_ik'_i\mathscr C_i\cdot \widetilde \sigma_C+\sum_{i=a+b+1}^{a+b+r+t}k'_i \widetilde \sigma_i\cdot \widetilde \sigma_C=lk'_0\Big(\sum_{i=0}^aD_i\cdot C+\sum_{i=a+b+1}^{a+b+r}\frac{1}{l_i}\overline S_i\cdot C\Big),$$
and again, the latter equality holds true since both sides vanish. Indeed, the left-hand side is now the intersection $\widetilde\sigma_C\cdot\mathrm{div} f^*(x^m)$ and the right-hand side is $lk'_0\widetilde B\cdot C$. This concludes the proof of \eqref{eq:balance claim} and of Case 1.

Case 2: {\em $\chi$ does not contract $\overline S$.} By the assumptions of the theorem, the graph $\GG_{S}$ is weightless and $3$-valent except for at most one $4$-valent vertex. Since rational curves with three marked points have no moduli, it follows that $\GG_{S}$ has a $4$-valent vertex, which we denote by $u\in\GG_{S}$. In this case we will show that the map $\alpha$ is locally combinatorially surjective at $W$. Consider $\widetilde \cC_u$ as above. By construction, $\widetilde \cC_u\to \overline S$ is a family of rational curves and since $\widetilde \cC \to \widetilde B$ is split, the four marked points in each fiber define sections of the family. Let $\xi\:\overline S \to \overline M_{0, 4}\simeq \PP^1$ be the induced map. Since $\chi$ does not contract $\overline S$, it follows that  $\xi$ is not constant, and hence surjective. Thus, the preimage of each boundary point of $\overline M_{0,4}$ in $\overline S$ is the union of closures of some codimension-one strata $S'$. Such strata correspond to faces $W'\in\mathcal F(\Lambda)$ containing $W$ as a codimension-one face. Since $\xi$ is surjective, it follows that for each weightless and $3$-valent types $\Theta'$ corresponding to one of the three possible resolutions of the $4$-valent vertex $u$, there exists a face $W'$ as above for which $\alpha(W')\cap M_{[\Theta']}\neq \emptyset$. Therefore $\alpha$ is locally combinatorially surjective along $W$ as claimed.
\qed

\bibliographystyle{amsalpha}
\bibliography{1}
\end{document}